\documentclass{amsart}
\usepackage{amssymb, amsmath, amsthm, amsfonts, epsf, epsfig, color, enumerate}
\usepackage{comment}
\newtheorem{theorem}{Theorem}
\newtheorem*{theorema}{Theorem}
\newtheorem*{theoremab}{Theorem B}
\newtheorem*{theoremac}{Theorem C}
\newtheorem*{theore}{Theorem A}

\newtheorem{lemma}{Lemma}
\newtheorem{corollary}{Corollary}

\theoremstyle{definition}

\newtheorem{remark}{Remark}
\allowdisplaybreaks

\theoremstyle{remark}

\begin{document}

\title{The Burgess inequality and the least $k$-th power non-residue}

\author{Enrique Trevi\~no }
\address{Department of Mathematics and Computer Science, Lake Forest College, Lake Forest, Illinois 60045}
\email{trevino@lakeforest.edu}
\thanks{Many results in this paper are in Chapter 6 of the author's Ph. D. Dissertation \cite{ET}.}

\subjclass[2010]{Primary 11L40, 11Y60}
\keywords{Character Sums, Burgess inequality, power residues and non-residues.}

\begin{abstract}
The Burgess inequality is the best upper bound we have for the character sum $S_{\chi}(M,N) = \sum_{M<n\le M+N} \chi(n).$ Until recently, no explicit estimates had been given for the inequality. In 2006, Booker gave an explicit estimate for quadratic characters which he used to calculate the class number of a 32-digit discriminant. McGown used an explicit estimate to show that there are no norm-Euclidean Galois cubic fields with discriminant greater than $10^{140}$. Both of their explicit estimates are on restricted ranges. In this paper we prove an explicit estimate that works for any $M$ and $N$. We also improve McGown's estimates in a slightly narrower range, getting explicit estimates for characters of any order. We apply the estimates to the question of how large must a prime $p$ be to ensure that there is a $k$-th power non-residue less than $p^{1/6}$.
\end{abstract}

\maketitle

\section{introduction}
Let $\chi$ be a character $\bmod{\,q}$ for some integer $q > 1$. Throughout the paper we will use the term character for Dirichlet character. Let $S_{\chi}(M,N)$ be defined as follows
\begin{equation*}
S_{\chi}(M,N) = \sum_{M<n\le M+N} \chi(n).
\end{equation*}
Historically, studying this sum has proven fruitful in analytic number theory to bound the least $k$-th power non-residue, to bound class numbers, to bound the least inert prime in a number field, and to bound the least primitive root, among other applications. The first non-trivial bound for this sum was proven independently by P\'olya and Vinogradov in 1918; namely, they showed that there exists an absolute constant $c>0$ such that $|S_{\chi}(M,N)|\le c\sqrt{q}\log{q}$. The P\'olya--Vinogradov inequality is very useful when $N$ is big compared to $\sqrt{q}$, but not very useful otherwise (since trivially $|S_{\chi}(M,N)| \le N$). What we want is an inequality that gives a nontrivial result even when $N$ is small compared to $\sqrt{q}$. The best theorem for short character sums is known as the Burgess inequality (\cite{Bur1957},\cite{Bur1962}, \cite{Bur1963}, \cite{Bur1986}) and allows us to take $N$ as small as $q^{\frac{1}{4}-o(1)}$.\footnote{To get $-o(1)$ instead of $+o(1)$, one needs a clever use of the large sieve as done by Hildebrand in \cite{Hil1986}.} We state the theorem below:
\begin{theorema}[Burgess]
Let $\chi$ be a primitive character $\bmod{\,q}$ with $q > 1$, and let $M$ and $N$ be non-negative reals with $N\ge 1$. Then $$|S_{\chi}(M,N)| \ll N^{1-\frac{1}{r}}q^{\frac{r+1}{4r^2} + \varepsilon}$$ for $r = 2,3$ and for any $r \geq 1$ if $q$ is cubefree, the implied constant depending only on $\varepsilon$ and $r$.
\end{theorema}

To prove the Burgess inequality, one of the keys is the following inequality which relies on a deep theorem of Weil \cite{Weil}:
\begin{theore}\label{main inequality Burgess}
For $p$ a prime number, $r$ a positive integer and $B$ a positive real number satisfying $r \le 9B$, let $\chi$ be a non-principal character to the modulus $p$. Then
$$\sum_{x \bmod{\,p}}\left|\sum_{1\le b\le B} \chi(x + b)\right|^{2r} \le (2r-1)!! B^{r} p + (2r-1)B^{2r}\sqrt{p}.,$$
where $(2r-1)!! = (2r-1)(2r-3)\ldots(1).$
\end{theore}
The above theorem was proven with weaker constants by Erd{\"o}s and Davenport in \cite{DE1952}, and Burgess improved it to better constants and used it to get the Burgess inequality. In \cite{Boo2006}, Booker proved it with these constants for quadratic characters. In \cite{ETk}, the author extended it to all characters. The reliance on the Weil estimate makes it difficult to improve the Burgess inequality asymptotically.

Recently, some problems have required getting explicit estimates on the Burgess inequality. Booker in \cite{Boo2006} needed an explicit form of the inequality to compute a 32-digit discriminant. McGown in \cite{McG3} used an explicit form of the inequality to show that there are no norm-Euclidean Galois cubic fields of discriminant greater than $10^{140}$. The goal of this paper is to improve their explicit estimates in the ranges they work in and give an explicit estimate that works regardless of the range of $N$. We apply these estimates to a question about $k$-th power non-residues $\bmod{\,p}$.

The work of Booker and McGown relies on the exposition of the Burgess inequality in \cite{IK2004}. In that book, Iwaniec and Kowalski sketch the proof of the following:
\begin{theoremab}\label{theorem Iwaniec}
Let $p$ be a large enough prime. Let $\chi$ be a non-principal character $\bmod{\,p}$. Let $r$ be a positive integer, and let $M$ and $N$ be non-negative integers with $N\ge 1$. Then
$$|S_{\chi}(M,N)| \le 30 N^{1-\frac{1}{r}} p^{\frac{r+1}{4r^2}}(\log{p})^{\frac{1}{r}}.$$
\end{theoremab}

In Section \ref{burgess section 1} we improve Theorem B to
\begin{theorem} \label{burgess kiks 1}
Let $p$ be a prime. Let $\chi$ be a non-principal character $\bmod{\,p}$. Let $M$ and $N$ be non-negative integers with $N\ge 1$, let $2\le r\le 10$ be a positive integer, and let $p_0$ be a positive real number. Then for $p \ge p_0$, there exists $c_1(r)$, a constant depending on $r$ and $p_0$ such that
$$|S_{\chi}(M,N)| \le c_1(r) N^{1-\frac{1}{r}} p^{\frac{r+1}{4r^2}}(\log{p})^{\frac{1}{r}},$$
where $c_1(r)$ is given by Table \ref{super table burgess}.
\begin{table}[h]
\begin{center}
	\begin{tabular}{|c| c | c| c|}
\hline
	$r$ & $p_0  = 10^7$ & $p_0= 10^{10}$ & $p_0= 10^{20}$ \\ \hline
2 & 2.7381 & 2.5173 & 2.3549 \\  \hline

3 & 2.0197 & 1.7385 & 1.3695 \\  \hline

4 & 1.7308 & 1.5151 & 1.3104 \\  \hline

5 & 1.6107 & 1.4572 & 1.2987 \\  \hline

6 & 1.5482 & 1.4274 & 1.2901 \\  \hline

7 & 1.5052 & 1.4042 & 1.2813 \\  \hline

8 & 1.4703 & 1.3846 & 1.2729 \\  \hline

9 & 1.4411 & 1.3662 & 1.2641 \\  \hline

10 & 1.4160 & 1.3495 & 1.2562 \\  \hline
	\end{tabular}
\end{center}
\caption{Values for the constant $c_1(r)$ in the Burgess inequality.}\label{super table burgess}
\end{table}
\end{theorem}

In the spirit of Theorem B, where we have no restriction on $r$, we also prove the following corollary:
\begin{corollary}\label{burgess corollary 1}
Let $p$ be a prime such that $p \ge 10^7$. Let $\chi$ be a non-principal character $\bmod{\,p}$. Let $r$ be a positive integer, and let $M$ and $N$ be non-negative integers with $N\ge 1$. Then
$$|S_{\chi}(M,N)| \le 2.74 N^{1-\frac{1}{r}} p^{\frac{r+1}{4r^2}}(\log{p})^{\frac{1}{r}}.$$
\end{corollary}

Restricting $N$ to be less than $4 p^{\frac{1}{2} + \frac{1}{4r}}$, McGown in \cite{McG3} proved an explicit version of Burgess with worse constants but with a better exponent in $\log{p}$. Indeed, he proved:
\begin{theoremac}\label{McGown's theorem burgess}
Let $p\ge 2\cdot 10^{4}$ be a prime number. Let $M$ and $N$ be non-negative integers with $1\le N\le 4 p^{\frac{1}{2} + \frac{1}{4r}}$. Suppose $\chi$ is a non-principal character $\bmod{\,p}$. Then there exists a computable constant $C(r)$ such that
\begin{equation*}
|S_{\chi}(M,N)| < C(r) N^{1-\frac{1}{r}} p^{\frac{r+1}{4r^2}}(\log{p})^{\frac{1}{2r}},
\end{equation*}
where $C(r)$ is given by Table \ref{table mcgown burgess}.
\begin{table}[h]
\begin{center}
	\begin{tabular}{|c| c | c| c|}
\hline
	$r$ & $C(r)$ & $r$ & $C(r)$ \\ \hline
	2 & 10.0366 & 9 & 2.1467 \\ \hline
	3 & 4.9539 & 10 & 2.0492 \\ \hline
	4 & 3.6493 & 11 & 1.9712 \\ \hline
	5 & 3.0356 & 12 & 1.9073 \\ \hline
	6 & 2.6765 & 13 & 1.8540 \\ \hline
	7 & 2.4400 & 14 & 1.8088 \\ \hline
	8 & 2.2721 & 15 & 1.7700\\ \hline
	
	\end{tabular}
\end{center}
\caption{Values for the constant $C(r)$ in the Burgess inequality.}\label{table mcgown burgess}
\end{table}
\end{theoremac}

The restriction that $N \le 4p^{\frac{1}{2} + \frac{1}{4r}}$ is used to get the exponent $\frac{1}{2r}$ in the $\log{p}$ term of the inequality. In Section \ref{burgess section 2}, we improve McGown's Theorem to have better constants in a similar range.
\begin{theorem}\label{burgess kiks 2}
Let $p$ be a prime. Let $\chi$ be a non-principal character $\bmod{\,p}$. Let $M$ and $N$ be non-negative integers with $1\le N\le 2 p^{\frac{1}{2} + \frac{1}{4r}}$, let $r\le 10$ be a positive integer, and let $p_0$ be a positive real number. Then for $p \ge p_0$, there exists $c_2(r)$, a constant depending on $r$ and $p_0$ such that
\begin{equation*}
|S_{\chi}(M,N)| < c_2(r) N^{1-\frac{1}{r}} p^{\frac{r+1}{4r^2}}(\log{p})^{\frac{1}{2r}},
\end{equation*}
where $c_2(r)$ is given by Table \ref{table kiks burgess}.

\begin{table}
\begin{center}
	\begin{tabular}{|c| c | c| c|}
\hline
	$r$ & $p_0 = 10^{10}$ & $p_0 = 10^{15}$ & $p_0 = 10^{20}$ \\ \hline
2 & 3.6529 & 3.5851 & 3.5751 \\  \hline

3 & 2.5888 & 2.5144 & 2.4945 \\  \hline

4 & 2.1914 & 2.1258 & 2.1078 \\  \hline

5 & 1.9841 & 1.9231 & 1.9043 \\  \hline

6 & 1.8508 & 1.7959 & 1.7757 \\  \hline

7 & 1.7586 & 1.7066 & 1.6854 \\  \hline

8 & 1.6869 & 1.6384 & 1.6187 \\  \hline

9 & 1.6283 & 1.5857 & 1.5654 \\  \hline

10 & 1.5794 & 1.5410 & 1.5216 \\  \hline
\end{tabular}
\end{center}\caption{Values for the constant $c_2(r)$ in the Burgess inequality.}\label{table kiks burgess}
\end{table}

\end{theorem}
Using an idea from \cite{MV2007}, we can get rid of the restriction on $N$ for $r\ge 3$.

\begin{corollary}\label{burgess kiks corollary 2}
Let $p\ge 10^{10}$ be a prime number. Let $M$ and $N$ be non-negative integers with $N\ge 1$. Suppose $\chi$ is a non-principal character $\bmod{\,p}$ and that $p \ge p_0$ for some positive real $p_0$. Then for $r\ge 3$, there exists a computable constant $c_2(r)$ depending on $r$ and $p_0$, such that
\begin{equation*}
|S_{\chi}(M,N)| < c_2(r) N^{1-\frac{1}{r}} p^{\frac{r+1}{4r^2}}(\log{p})^{\frac{1}{2r}},
\end{equation*}
where $c_2(r)$ is the same as that of Table \ref{table kiks burgess} whenever $r\ge 3$.
\end{corollary}

Putting an extra restriction on $N$ (namely, $N \le 2\sqrt{p})$, Booker in \cite{Boo2006} gave better bounds in the special case of quadratic characters.

\begin{remark}
Using Theorem A one could extend Booker's theorem to all orders of $\chi$ (with slightly worse constants). The reason we would get slightly worse constants is that in the quadratic case, the inequality in Theorem A can be improved slightly to $(2r-1)!! B^{r} p + (2r-2)B^{2r}\sqrt{p}$. Every other part of Booker's proof extends naturally, but this part of the inequality fails when looking at higher orders.
\end{remark}

In Section \ref{section p^1/6}, we apply these estimates to a question about $k$-th power non-residues $\bmod{\,p}$. Indeed, let $p$ be a prime and let $k$ be an integer with $k \mid p-1$ and $k > 1$. Let $g(p,k)$ be the least $k$-th power non-residue $\bmod{\,p}$. The case $k = 2$, i.e., the question of how big the least quadratic non-residue is, has been studied extensively. A probabilistic heuristic using that a prime $q$ is a quadratic non-residue $\bmod{\,p}$ half of the time, suggests that $g(p,2) \ll \log{p}\log{\log{p}}$ and that $g(p,2) \gg \log{p}\log{\log{p}}$ for infinitely many $p$. Assuming the Generalized Riemann Hypothesis for Dirichlet L-functions (GRH), Ankeny \cite{Ank1952} showed that $g(p,k) \ll (\log{p})^2$ and Bach \cite{Bach1990} made this explicit by proving (under GRH) that $g(p,k) \le 2(\log{p})^2.$ The best unconditional results (for $g(p,k)$) are due to Burgess \cite{Bur1957}, who, building on work by Vinogradov \cite{Vinogradov}, showed that $$g(p,k) \ll_{\varepsilon} p^{\frac{1}{4\sqrt{e}}+\varepsilon}.$$

For $k \ge 3$, better estimates which depend upon $k$ have been proven by Wang Yuan(\cite{Yuan}) building on work of Vinogradov (\cite{Vinogradov}) and Buh{\v{s}}tab (\cite{Bu1949}).

All of the work described so far has been of asymptotic nature. In terms of getting explicit bounds, Karl Norton \cite{Nor1971}, building on a technique of Burgess \cite{Bur19632}, was able to show that $g(p,k) \leq 3.9 p^{1/4}\log p$ unless $k =2$ and $p\equiv 3 \pmod 4$ for which he showed  $g(p,k) \leq 4.7 p^{1/4}\log p$. In \cite{ETk}, the author improved Norton's bounds to $0.9 p^{1/4}\log{p}$ and $1.1 p^{1/4}\log{p}$, respectively. These bounds are far from the asymptotic bound of $p^{\frac{1}{4\sqrt{e}}+\varepsilon}$. In this paper, as an application of the explicit Burgess inequality, we find an upper bound on how large $p$ has to be to ensure that there is a $k$-th power non-residue less than $p^{1/6}$.

\begin{theorem}\label{theorem p^1.6}
Let $p$ be a prime number and $k >1$ be a positive divisor of $p-1$. Then for $p \ge 10^{4732}$, the least $k$-th power non-residue $\bmod{\,p}$ is less than or equal to $p^{1/6}$.
\end{theorem}

\begin{remark}
The techniques involved in the proof can be used to answer this question for $p^{\alpha}$ whenever $\alpha > \frac{1}{4\sqrt{e}}$.
\end{remark}

\section{Preliminary lemmas}
Let $A$ and $N$ be positive integers. Let $v(x)$ be the number of representations of $x$ as $\bar{a}n \pmod{p}$, where $\bar{a}$ is the inverse of $a \pmod{p}$,  $1\le a\le A$ and $M < n \le M+N$, that is,
\begin{equation}\label{burgess v(x)}
v(x) = \# \left\{(a,n)\in\mathbb{N}^2\mid 1\le a\le A, \mbox{ } M < n\le M+N \mbox{ and } n\equiv ax \bmod{\,p}\right\}.
\end{equation}

The main lemma in this section is the following:
\begin{lemma}\label{burgess V2 lemma}
Let $p$ be a prime and let $N < p$ be a positive integer. Let $A\ge 28$ be an integer satisfying $A< \frac{N}{12}$, then
\begin{equation}\label{burgess V2 claim}
V_2 = \sum_{x\kern-3pt\mod{p}}v^2(x) \leq 2AN\left(\frac{AN}{p} + \log(1.85 A)\right).
\end{equation}
\end{lemma}

To prove the lemma regarding $V_2$ we will need a couple of estimates involving the $\phi$ function (Lemmas \ref{burgess claim 2} and \ref{burgess claim 1}), an estimate on a sum of logarithms (Lemma \ref{burgess claim 3}) and a non-trivial combinatorial count (Lemma \ref{burgess counting lemma}).

\begin{lemma}\label{burgess claim 2}
For $x\ge 1$ a real number we have:
\begin{equation}\label{burgess claim 2 eq}
\sum_{n\leq x}\frac{\phi(n)}{n} \leq \frac{6}{\pi^2} x + \log{x} + 1.
\end{equation}
\end{lemma}
\begin{proof} For $1\le x < 2$, the left hand side of \eqref{burgess claim 2 eq} is $1$, while the right hand side is at least $1$. We can manually check that for all integers $x$ satisfying $2\le x \le 41$ we have
$$\sum_{n\le x} \frac{\phi(n)}{n} \le \frac{6}{\pi^2}(x-1) + \log{(x-1)} + 1,$$
implying that \eqref{burgess claim 2 eq} is true for $x< 41$. Therefore, we may assume that $x \ge 41$.

Let's work with the sum:
\begin{equation}\label{kakaro}
\sum_{n\leq x}\frac{\phi(n)}{n} = \sum_{n \leq x}\frac{1}{n}\sum_{d | n}\mu(d)\frac{n}{d} = \sum_{d\leq x}\sum_{n\leq \frac{x}{d}}\frac{\mu(d)}{d} = \sum_{d\leq x}\left\lfloor\frac{x}{d}\right\rfloor\frac{\mu(d)}{d}.
\end{equation}

From \cite[Theorem $422$]{Hardy} it follows that for $x \ge 1$
\begin{equation}\label{Hardylin}
\sum_{d\le x}\frac{1}{d} < \log{x} + \gamma + \frac{1}{x}.
\end{equation}

Using \eqref{Hardylin} in \eqref{kakaro} yields
\begin{equation}\label{almost there 1}
\sum_{n\le x}\frac{\phi(n)}{n} \le x\sum_{d\ge 1} \frac{\mu(d)}{d^2} -x\sum_{d>x}\frac{\mu(d)}{d^2}+ \sum_{d\leq x} \frac{1}{d} \leq\frac{6}{\pi^2}x + \log{x} + \gamma+\frac{1}{x} - x\sum_{d>x}\frac{\mu(d)}{d^2}.
\end{equation}

Moser and Macleod \cite{MM1966} gave a simple proof that for $x\ge 2$ we have
\begin{equation}\label{moser}
\left|\sum_{d>x}\frac{\mu(d)}{d^2}\right|\le \frac{1}{3x} + \frac{8}{3x^2}.
\end{equation}

Combining \eqref{moser} with \eqref{almost there 1} yields for $x\ge 41$ that
$$\sum_{n\le x}\frac{\phi(n)}{n} \le \frac{6}{\pi^2}x + \log{x} + \gamma + \frac{1}{x}+\frac{1}{3} + \frac{8}{3x} \le \frac{6}{\pi^2}x + \log{x} + 1.$$
\end{proof}

\begin{lemma}\label{burgess claim 1}
For $x\ge 1$ a real number we have:
\begin{equation}\label{burgess claim 1 eq}
\sum_{n \leq x} n\phi(n) \leq \frac{2}{\pi^2}x^3 + \frac{1}{2}x^2\log{x} + x^2.
\end{equation}
\end{lemma}
\begin{proof}
For $1\le x<2$, the left hand side of \eqref{burgess claim 1 eq} is $1$, while the right hand side is at least $x^2 \ge 1$. Therefore it is true for $1\le x<2$. Now for $2\le x<3$, the left hand side is $3$, while the right hand side is at least $x^2 \ge 4$. Therefore \eqref{burgess claim 1 eq} is true for $1\le x < 3$. In the rest of the proof we will assume that $x\ge 3$.
 Let's work with the sum:
\begin{multline*}
\sum_{n\leq x} \phi(n)n = \sum_{n \leq x}\sum_{d | n} \frac{\mu(d)n^2}{d} = \sum_{d \leq x} \mu(d)d \sum_{dm \leq x}m^2\\ = \sum_{d\leq x}\frac{\mu(d)d}{6}\left\lfloor\frac{x}{d}\right\rfloor\left(\left\lfloor\frac{x}{d}\right\rfloor + 1\right)\left(2\left\lfloor\frac{x}{d}\right\rfloor + 1\right).
\end{multline*}

Now, let $\theta_d =\frac{x}{d}-\left\lfloor\frac{x}{d}\right\rfloor.$ Then we have  \begin{multline}\label{choco2}
\sum_{n\le x} \phi(n) n = \frac{x^3}{3}\sum_{d\leq x} \frac{\mu(d)}{d^2}  + \frac{x^2}{6}\sum_{d\le x}\frac{(3-6\theta_d )\mu(d)}{d} \\+ \frac{x}{6}\sum_{d\le x}\left(6\theta_d^2 - 6\theta_d + 1\right)\mu(d) - \frac{1}{6}\sum_{d\le x} \theta_d(1-\theta_d )(1-2\theta_d )\mu(d)d .
\end{multline}

From \eqref{Hardylin} it follows that for $x\ge 3$
\begin{equation}\label{hardy no mames}
\sum_{d\le x} \frac{1}{d} < \log{x} + \gamma + \frac{1}{x} < \log{x} + 1-\frac{1}{60}-\frac{1}{60x}.
\end{equation}

Using that $0\le\theta_d \leq 1$ we have that $|3-6\theta_d|\le 3$, that $|6\theta_d^2 - 6\theta_d+1| \le 1$ and $|(1-\theta_d)(1-2\theta_d)(-\theta_d)| \le \frac{1}{10}$. Therefore, using \eqref{choco2}, \eqref{hardy no mames}, that $\displaystyle\sum_{d\ge 1}\frac{\mu(d)}{d^2} = \frac{6}{\pi^2}$, and that $|\mu(d)| \leq 1$,  we get
\begin{multline}\label{burgess lemma 2 eq}
\sum_{n\le x} \phi(n) n \leq \frac{x^3}{3}\sum_{d\leq x} \frac{\mu(d)}{d^2}  + \frac{x^2}{2}\sum_{d\leq x}\frac{1}{d} + \frac{x}{6}\sum_{d\leq x} 1 +  \frac{1}{60}\sum_{d \leq x} d \\
\leq \frac{2}{\pi^2}x^3 - \frac{x^3}{3}\sum_{d>x}\frac{\mu(d)}{d^2} + \frac{1}{2}x^2\log{x} + \frac{x^2}{2} -\frac{x^2}{120} -\frac{x}{120}+\frac{x^2}{6}+ \frac{1}{60}\left(\frac{x(x+1)}{2}\right).
\end{multline}

From \eqref{moser} we have (for $x \ge 2$)
\begin{equation*}\label{burgess key to lemma 2}
\sum_{d> x}\frac{\mu(d)}{d^2} \ge \frac{1}{3x} - \frac{8}{3x^2} \ge -\frac{1}{x}.
\end{equation*}
Combining this with \eqref{burgess lemma 2 eq} yields the lemma.
\end{proof}

\begin{lemma}\label{burgess claim 3}
 For $x \geq 1$ we have:
$$\sum_{d\leq x} \log{\left(\frac{x}{d}\right)} \leq x -1$$
\end{lemma}
\begin{proof} For $1\leq x <2$ we have $\displaystyle\sum_{d\leq x} \log{\left(\frac{x}{d}\right)} = \log{x} \leq x -1$. Therefore, we may assume $x \geq 2$. Now,
\begin{equation}\label{logsum}
\sum_{d\leq x} \log{\left(\frac{x}{d}\right)} = \left\lfloor x\right\rfloor \log{x} - \sum_{d\leq x} \log{d} \leq \left\lfloor x\right\rfloor \log{x} - \left\lfloor x\right\rfloor \log{\lfloor x\rfloor}+\lfloor x\rfloor -1.
\end{equation}
To get the second inequality we used that $$\sum_{d\leq x} \log{d} = \sum_{2\le d\le x}\log{d} \geq \int_{1}^{\lfloor x\rfloor}\log{t}\,dt = \lfloor x \rfloor\log{\lfloor x\rfloor} -\lfloor x\rfloor + 1.$$

Now, notice that $x = \lfloor x \rfloor + \{x\}$ and $\log{(1+y)}\leq y$. Therefore we have
\begin{equation}\label{logito}
\lfloor x\rfloor \log{x} = \lfloor x\rfloor\log{\lfloor x\rfloor} + \lfloor x\rfloor\log{(x/\lfloor x\rfloor)}\leq \lfloor x\rfloor\log{\lfloor x\rfloor} + \{x\}.
\end{equation}

Combining equations (\ref{logsum}) and (\ref{logito}) yields

\begin{equation*}
\sum_{d\leq x} \log{\left(\frac{x}{d}\right)}\leq \{x\} + \lfloor x\rfloor - 1 = x-1.
\end{equation*}
\end{proof}

\begin{lemma}\label{burgess counting lemma}
Let $A \ge 2$, $N \ge 1$, $a_1$, $a_2$ and $M$ be integers. Let $p > N$ be a prime number. Suppose, $1\le a_1,a_2 \le A$ with $a_1 \neq a_2$. Then the number of pairs of integers $(n_1,n_2)$ satisfying $M < n_1,n_2 \le N+M$ and $a_1 n_2 - a_2 n_1 = kp$ is bounded above by
$$N \frac{\gcd{(a_1,a_2)}}{\max\{a_1,a_2\}} + 1.$$
\end{lemma}

\begin{proof}
Let $d = \gcd{(a_1,a_2)}$. Since $a_1n_2 -a_2n_1 = kp$, we have that $d | k$. Let $a_1=a_1^{'}d, a_2 = a_2^{'}d$ and $k = k^{'} d$. Now, we also have
\begin{equation}\label{goku}
n_2 = \frac{k p+a_2n_1}{a_1} = \frac{k^{'} p + a_2^{'}n_1}{a_1^{'}}.
\end{equation}
The right hand side of \eqref{goku} must be an integer. Therefore $k^{'}p + a_2^{'}n_1 \equiv 0 \bmod{a_1^{'}}$. Since this is a linear equation in terms of $n_1$, there is at most one solution $\bmod{\,a_1^{'}}$. Therefore, in the interval $(M, M+N]$ there are at most $\displaystyle\frac{N}{a_1^{'}} + 1$ choices of $n_1$. Since $n_2$ is uniquely determined from $n_1$, the number of pairs $(n_1, n_2)$ satisfying the conditions of the lemma is bounded by
$$\frac{N}{a_1^{'}} + 1 = N\frac{\gcd{(a_1,a_2)}}{a_1} + 1.$$
Analogously, the number of pairs is bounded by $N\displaystyle\frac{\gcd{(a_1,a_2)}}{a_2} + 1$. The statement of the lemma is now an easy consequence.
\end{proof}

Now we are ready to prove Lemma \ref{burgess V2 lemma}.

\begin{proof}[Proof of Lemma \ref{burgess V2 lemma}]
We'll begin by noting that $V_2$ is the number of quadruples $(a_1,a_2,n_1,n_2)$ with $1\leq a_1,a_2 \leq A$ and $M < n_1,n_2 \leq M + N$ such that $a_1n_2 \equiv a_2n_1\pmod p$. If $a_1 = a_2$, since $N < p$, we have that $n_1 = n_2$ because $n_1 \equiv n_2 \pmod p$ while $|n_1 - n_2| \leq N < p$. Therefore, the number of quadruples in this case is $AN$. Fix $a_1$ and $a_2$ with $a_1 \neq a_2$. Let $k$ be an integer satisfying
\begin{equation}\label{burgess the equation}
a_1n_2 - a_2n_1 = kp,
\end{equation}
for some $n_1$ and $n_2$. We can put a bound on possible values for $k$. First of all, $k$ must be a multiple of $\gcd{(a_1,a_2)}$. Now, if we write $n_1 = n_1^{'} + M$ and $n_2 = n_2^{'} + M$, we have, using $kp - (a_1-a_2)M = a_1n_2^{'} - a_2n_1^{'}$, that
$$-a_2 N \le - a_2n_1^{'} < kp - (a_1-a_2)M < a_1n_2^{'} \le a_1 N  .$$
Therefore $k$ lies in an interval of length at most $\frac{(a_1 + a_2) N}{p}$. Since $k$ is a multiple of $\gcd{(a_1,a_2)}$ and $k$ lies in such an interval, then there are at most
$$\frac{(a_1 + a_2) N}{\gcd{(a_1,a_2)} p} + 1,$$
choices for $k$.

Given $a_1,a_2$ and $k$ we can count the number of pairs $(n_1,n_2)$ which would satisfy (\ref{burgess the equation}). By Lemma \ref{burgess counting lemma}, the number of pairs is bounded by $N \frac{\gcd{(a_1,a_2)}}{\max\{a_1,a_2\}} + 1$. Therefore we get
\begin{multline}\label{bubu}
V_2 \leq AN + 2\sum_{a_1 < a_2} \Big( \frac{(a_1 + a_2) N}{\gcd{(a_1,a_2)}p} + 1   \Big) \Big( \frac{\gcd{(a_1,a_2)} N}{\max\{a_1,a_2\}} + 1  \Big) \\ = AN + \frac{2N^2}{p} S_1 + \frac{2N}{p} S_2  + 2N S_3 + A^2 - A,
\end{multline}
where
\begin{equation*}\label{S_1}
S_1 = \sum_{a_1 < a_2} \frac{a_1 + a_2}{\max\{a_1,a_2\}} \mbox{ , }
\end{equation*}
\begin{equation*}\label{S_2}
S_2 = \sum_{a_1 < a_2} \frac{a_1 + a_2}{\gcd{(a_1,a_2)}} \mbox{ , }
\end{equation*}
and
\begin{equation}\label{S_3}
S_3 = \sum_{a_1 < a_2} \frac{\gcd{(a_1,a_2)}}{\max\{a_1,a_2\}}.
\end{equation}

Dealing with $S_1$ is straightforward, in fact $S_1$ is
\begin{equation}\label{burgess S1}
\sum_{a_2 \leq A}\sum_{a_1 < a_2} \frac{a_1 + a_2}{a_2} = \sum_{a_2 \leq A} \left(a_2 -1 + \frac{a_2(a_2-1)}{2a_2}\right) = \frac{3}{2}\frac{A(A-1)}{2} = \frac{3}{4}A^2- \frac{3}{4}A.
\end{equation}

Now, let's estimate $S_2$:
\begin{multline*}
S_2 = \sum_{a_1< a_2 \leq A} \frac{a_1+a_2}{\gcd{(a_1,a_2)}} = \sum_{d \leq A}\sum_{b_2 \leq \frac{A}{d}}\sum_{b_1 < b_2, (b_1,b_2) = 1} \left(b_1 + b_2\right)\\ = \sum_{d\leq A}\sum_{2\leq b_2 \leq \frac{A}{d}} \left(\frac{\phi(b_2)}{2}b_2+ \phi(b_2)b_2 \right) = \frac{3}{2}\sum_{d\leq A}\sum_{2\leq b_2 \leq \frac{A}{d}} \phi(b_2)b_2.
\end{multline*}
Using Lemma \ref{burgess claim 1}, we get
\begin{equation*}
S_2 \leq \frac{3}{\pi^2}\sum_{d\leq A}\big(\frac{A}{d}\big)^3 + \frac{3}{4}\sum_{d\leq A} \big(\frac{A}{d}\big)^2\log{(\frac{A}{d})} + \frac{3}{2}\sum_{d\leq A} \big(\frac{A}{d}\big)^2.
\end{equation*}
Using that $\log{(\frac{A}{d})} = \log{A}-\log{d}$,  and that $\displaystyle\sum_{d\ge 1} \frac{1}{d^s} = \zeta(s)$, yields
\begin{equation*}
S_2 \le \frac{3\zeta(3)}{\pi^2}A^3 + \frac{3\zeta(2)}{4}A^2\log{A} - \frac{3}{4}A^2\sum_{d\le A}\frac{\log{d}}{d^2} + \frac{3}{2}A^2\zeta(2).
\end{equation*}
Using that for $A\ge 11$ we have $\frac{3\zeta(2)}{2} - \frac{3}{4}\sum_{d\le A}\frac{\log{d}}{d^2} < 2$ yields
\begin{equation}\label{burgess S2}
S_2 \leq \frac{3\zeta(3)}{\pi^2}A^3 + \frac{3\zeta(2)}{4}A^2\log{(A)} + 2 A^2.
\end{equation}

Let's estimate $S_3$. We have
$$S_3 = \sum_{a_1< a_2 \leq A} \frac{\gcd{(a_1,a_2)}}{\max(a_1,a_2)} = \sum_{d \leq A}\sum_{b_2 \leq \frac{A}{d}}\sum_{b_1 < b_2, (b_1,b_2) = 1} \frac{1}{b_2} = \sum_{d \leq A} \sum_{2\leq b_2 \leq \frac{A}{d}} \frac{\phi(b_2)}{b_2}.$$
Using Lemma \ref{burgess claim 2} yields
$$S_3 \leq \sum_{d\leq A} \left(\frac{A}{d}\frac{1}{\zeta(2)} + \log{(\frac{A}{d})}\right) = \frac{6}{\pi^2} A\sum_{d\le A}\frac{1}{d} + \sum_{d\le A} \log{(\frac{A}{d})}. $$

From \eqref{Hardylin} it follows that for $A\ge 27$
\begin{equation*}
\sum_{d\le A} \frac{1}{d} < \log{A} + \gamma + \frac{1}{A} < \log{(1.85 A)}.
\end{equation*}
Using this and Lemma \ref{burgess claim 3} yields
\begin{equation}\label{burgess S3}
S_3 \le \frac{6}{\pi^2} A\log{(1.85A)} + A -1.
\end{equation}

 Using (\ref{burgess S1}), (\ref{burgess S2}) and (\ref{burgess S3}) in (\ref{bubu}) yields the following upper bound for $V_2$:
\begin{equation}\label{*}
2AN\Big(\frac{3}{2} + \frac{A-1}{2N} + \frac{3AN}{4p} - \frac{3N}{4p} + \frac{3\zeta(3)A^2}{\pi^2 p} + \frac{3\zeta(2)A\log{A}}{4p} + \frac{6}{\pi^2}\log{(1.85A)} -\frac{1}{A} + \frac{2A}{p}\Big)
\end{equation}

For $A\ge 4$, we have
\begin{equation}\label{burgess *1}
\frac{3\zeta(3)A^2}{\pi^2 p} + \frac{3\zeta(2)A\log{A}}{4p} < \frac{3}{4}\frac{A^2}{p}.
\end{equation}
Since $N > 3A$ we have the following two inequalities:
\begin{equation}\label{1}\frac{AN}{4p} > \frac{3}{4}\frac{A^2}{p} \mbox{ and } \frac{3N}{4p} > \frac{2A}{p}.
\end{equation}

Combining \eqref{burgess *1} and \eqref{1} yields
\begin{equation}\label{burgess *2}
\frac{3AN}{4p} +\left(\frac{3\zeta(3)A^2}{\pi^2 p} + \frac{3\zeta(2) A\log{A}}{4p}\right)+\left(\frac{2A}{p} - \frac{3N}{4p}\right) < \frac{AN}{p}.
\end{equation}
Finally, using that $A\ge 28$ and that $N > 12A$, yields
\begin{equation}\label{3}\Big(1-\frac{6}{\pi^2}\Big)\log{(1.85 A)} \geq \Big(1-\frac{6}{\pi^2}\Big)\log{(51.8)} \geq 1.54766 > \frac{3}{2} + \frac{1}{24}  \geq \frac{3}{2} + \frac{A}{2N} .\end{equation}

Combining \eqref{burgess *2} and \eqref{3} in \eqref{*} yields \eqref{burgess V2 claim}.
\end{proof}

\begin{remark} The main term will come from the $\log{(1.85A)}$ term and the $1.85$ can be changed to a smaller number (the limit being $e^{\gamma}$), forcing $A$ to be slightly larger to make the inequalities work. Also, the coefficient on $\log{(1.85A)}$ can be changed to be as close to $\frac{6}{\pi^2}$ as we want as long as $A$ is big enough. It is important to note that big $A$'s will mean forcing $p$ to be much bigger in the estimates for the Burgess inequality.
\end{remark}

\begin{remark}
The constraint $A\ge 28$ is used to get the main term to be $\log{(1.85 A)}$; however, we can relax the condition on $A$ and get a slightly worse main term. We chose our values this way to get the constants in Tables \ref{super table burgess} and \ref{table kiks burgess} as low as possible for small values of $r$. Relaxing the $A\ge 28$ condition would make these constants worse, but improve the constants for larger values of $r$. Since the small values of $r$ seem to be the most useful in applications, we decided to focus on minimizing these cases.
\end{remark}

\section{Explicit Burgess inequality}\label{burgess section 1}
\begin{proof}[Proof of Theorem \ref{burgess kiks 1}]
Let $M$ and $N\ge 1$ be non-negative integers. Let $r \ge 2$ be a positive integer. Fix a constant $c_1(r) \ge 1$ (which we will name later). We will prove the Theorem by induction. Assume that for all positive integers $h< N$, we have
$$|S_{\chi}(M,h)|\le c_1(r) h^{1-\frac{1}{r}} p^{\frac{r+1}{4r^2}}(\log{p})^{\frac{1}{r}}.$$
The idea is to estimate $S_{\chi}(M,N)$ by shifting by $h$ ($n \mapsto n + h$) and getting an error that we can deal with by induction.

Note that, for all $h < N$,
\begin{equation*}
S_{\chi}(M,N) = \sum_{M < n \leq M+N} \chi(n + h) + \sum_{M<n\leq M+h} \chi(n) - \sum_{M+N < n \leq M+N + h} \chi(n).
\end{equation*}
Therefore
\begin{equation*}
S_{\chi}(M,N)  = \sum_{M < n \leq N+M} \chi(n + h)+ 2\theta E(h),
\end{equation*}
where $|\theta|\le 1$ which depends upon $h$, and $E(h) = \displaystyle\max_{K}|S_{\chi}(K,h)|$.

Let $A$ and $B$ be positive reals and let $H = \lfloor A\rfloor\lfloor B\rfloor$. We will use shifts of length $h = ab$ where $a$ and $b$ are positive integers satisfying $a \leq A$ and $b \leq B$. After averaging over all the pairs $(a,b)$ we get
\begin{equation}\label{bur eq 1}
S_{\chi}(M,N) = \frac{1}{H}\sum_{a,b} \sum_{M < n \leq M + N} \chi(n + ab) + \frac{1}{H}\sum_{a,b}2\theta E(ab).
\end{equation}
Let $v(x)$ be defined as in (\ref{burgess v(x)}), then
\begin{equation}\label{bur eq 2}
\left|\sum_{a,b}\sum_{M < n \leq M + N} \chi(n + ab)\right| \le \sum_{x\kern-3pt\mod{p}} v(x)\left|\sum_{b\le B} \chi(x+b)\right|.
\end{equation}
Let $$V:= \sum_{x\kern-3pt\mod{p}} v(x)\left|\sum_{b\le B}\chi(x+b)\right|.$$
Then, combining (\ref{bur eq 1}) with (\ref{bur eq 2}), we get
\begin{equation}\label{burgess ineq V/H + 2/H}
|S_{\chi}(M,N)| \leq \frac{V}{H} + \frac{2}{H}\sum_{a,b} E(ab).
\end{equation}

We can now focus on estimating $V$. Now define $V_1 := \displaystyle \sum_{x \pmod p} v(x)$,\\ $V_2 := \displaystyle \sum_{x \pmod p} v^2(x)$ and $W := \displaystyle \sum_{x \pmod p} \Big|\displaystyle \sum_{1 \leq b \leq B} \chi(x + b)\Big|^{2r}$. Using H\"older's Inequality we get
\begin{equation}\label{burgess V}
V \leq V_1^{1- \frac{1}{r}}V_2^{\frac{1}{2r}}W^{\frac{1}{2r}}.
\end{equation}

First note that
\begin{equation*}\label{burgess V1}
V_1 = \lfloor A\rfloor N\le AN.
\end{equation*}
From Lemma \ref{burgess V2 lemma}, for $\lfloor A\rfloor\ge 28$ and $\lfloor A\rfloor<\frac{N}{12}$, we have
\begin{equation}\label{burgess V2}
V_2  \leq 2AN\left(\frac{AN}{p} + \log(1.85 A)\right).
\end{equation}
We can also bound $W$, since by Theorem A, we have (for $r\le 9B$):
\begin{equation}\label{burgess W}
W \le \frac{(2r)!}{2^r r!}B^r p + (2r-1)B^{2r}\sqrt{p} = (2r-1)!! B^r p + (2r-1) B^{2r}\sqrt{p}.
\end{equation}

Let's head back to proving the Burgess bound. We will let $AB = kN$ for $k$ a real number to be chosen later. Using the inequalities of $V_1,V_2$ and $W$ together with \eqref{burgess V} yields the following bound upper bound for $\frac{V}{H}$:
\begin{multline}\label{Opti}
\frac{V}{H} \le \frac{1}{\lfloor A\rfloor \lfloor B\rfloor}V_1^{1-\frac{1}{r}}V_2^{\frac{1}{2r}}W^{\frac{1}{2r}} \leq \frac{\frac{AB}{\lfloor A\rfloor\lfloor B\rfloor} }{(AB)^{\frac{1}{2r}}}\cdot\frac{(2WB)^{\frac{1}{2r}}}{B}\left(\frac{AN}{p} + \log{(1.85A)}\right)^{\frac{1}{2r}}N^{1-\frac{1}{2r}}
\\\le \frac{A}{A-1}\cdot\frac{B}{B-1}\cdot\frac{1}{k^{\frac{1}{2r}}}\cdot\frac{(2WB)^{\frac{1}{2r}}}{B}\left(\frac{AN}{p} + \log{(1.85A)}\right)^{\frac{1}{2r}}N^{1-\frac{1}{r}}.
\end{multline}

Because of \eqref{Opti} we can see that a good choice for $B$ is the one that minimizes $\frac{WB}{B^{2r}}$. Using \eqref{burgess W}, we seek to minimize the expression  $(2r-1)!!\frac{p}{B^{r-1}} + (2r-1)Bp^{\frac{1}{2}}$. We take the derivative with respect to $B$ and equal it to zero. After this process we get that a good $B$ is
\begin{equation}\label{burgess B}
B = \big((2r-3)!!(r-1)\big)^{\frac{1}{r}}p^{\frac{1}{2r}}.
\end{equation}

Using this value of $B$ we get
\begin{equation}\label{WB}
\frac{(2WB)^{\frac{1}{2r}}}{B} \leq \left(\frac{2r(2r-1)}{r-1}\right)^{\frac{1}{2r}}(r-1)^{\frac{1}{2r^2}}\big((2r-3)!!\big)^{\frac{1}{2r^2}}p^{\frac{r+1}{4r^2}}.
\end{equation}

Now we must try to bound $\frac{AN}{p} + \log{(1.85A)}$. To do this, we can use the P\'olya--Vinogradov inequality to give an upper bound for $N$, since for $N$ large, the P\'olya--Vinogradov inequality would be a better bound than the Burgess inequality. Indeed, if
\begin{equation}\label{range}
N \geq p^{\frac{1}{2} + \frac{1}{4r}}\log{p},
\end{equation}
then, since $c_1(r)\ge 1$, we have $$c_1(r)N^{1-\frac{1}{r}}p^{\frac{r+1}{4r^2}}(\log{p})^{\frac{1}{r}} \ge \sqrt{p}\log{p}.$$
Therefore, from the P\'olya--Vinogradov inequality (see Section $9.4$ in \cite{MV2007}) we can conclude that $|S_{\chi}(M,N)|\le c_1(r)N^{1-\frac{1}{r}}p^{\frac{r+1}{4r^2}}\left(\log{p}\right)^{\frac{1}{r}},$ whenever we have \eqref{range}.

If we have $r\ge 3$, then we can use the Burgess inequality with $r-1$ instead of the P\'olya--Vinogradov inequality, to get a better upper bound on $N$. Indeed, if we let $s$ be a real number that satisfies
\begin{equation}\label{s condition}
c_1(r-1) \le s^{\frac{1}{r(r-1)}} c_1(r),
\end{equation}
then if
$$N\ge s\, p^{\frac{1}{4} + \frac{1}{2r} + \frac{1}{4r(r-1)}}\log{p},$$
then
$$c_1(r)N^{1-\frac{1}{r}}p^{\frac{r+1}{4r^2}}(\log{p})^{\frac{1}{r}}\ge c_1(r-1)N^{1-\frac{1}{r-1}}p^{\frac{r}{4(r-1)^2}}(\log{p})^{\frac{1}{r-1}}.$$

Similarly, we can put a lower bound on $N$, by noting that $|S_{\chi}(M,N)| \le N$. Indeed,
$$c_1(r)N^{1-\frac{1}{r}}p^{\frac{r+1}{4r^2}}(\log{p})^{\frac{1}{r}} \ge N,$$
whenever
\begin{equation*}\label{range low}
N\le c_1(r)^r p^{\frac{1}{4} + \frac{1}{4r}}\log{p}.
\end{equation*}

Therefore, we may assume that
\begin{equation}\label{range for N}
c_1(2)^2 p^{\frac{3}{8} }\log{p} < N < p^{\frac{5}{8}}\log{p},
\end{equation}
for $r=2$, and that
\begin{equation}\label{range2 for N}
c_1(r)^r p^{\frac{1}{4} + \frac{1}{4r}}\log{p} < N < s\, p^{\frac{1}{4} + \frac{1}{2r} + \frac{1}{4r(r-1)}}\log{p},
\end{equation}
for $r\ge 3$.

Using that $A = \frac{kN}{B}$, the upper bound for $N$ in \eqref{range for N}, and \eqref{burgess B}, we get \begin{equation}\label{log}
\frac{AN}{p} = \frac{kN^2}{pB} \le \frac{k p^{\frac{5}{4}}\log^2{p}}{p B} \le k\log^2{p} ,
\end{equation}
for $r=2$, and for $r\ge 3$, we get
\begin{equation}\label{log2}
\frac{AN}{p} = \frac{kN^2}{pB} \le \frac{s^2 k p^{\frac{1}{2} +\frac{1}{r}+\frac{1}{2r(r-1)}}\log^2{p}}{p B} = \frac{s^2 k}{\left((2r-3)!!(r-1)\right)^{\frac{1}{r}}p^{\frac{1}{2}-\frac{1}{2r}-\frac{1}{2r(r-1)}}}\log^2{p} ,
\end{equation}

Now we consider what happens to $\log{(1.85A)}$.
\begin{equation}\label{la otra vez}
 \log{(1.85A)} = \log{\left(\frac{1.85kN}{B}\right)} \leq \log{\left(1.85k\log{p}\right)}+\frac{3\log{p}}{8},
 \end{equation}
for $r=2$, and for $r\ge 3$, we get
\begin{equation}\label{la otra vez2}
 \log{(1.85A)} = \log{\left(\frac{1.85kN}{B}\right)} \leq \log{\left(\frac{1.85 s\,k\log{p}}{((2r-3)!!(r-1))^{\frac{1}{r}}}\right)}+\frac{\log p}{4} + \frac{\log{p}}{4r(r-1)} .
 \end{equation}

Now, let's bound the error term, the part we have labeled as $E(h)$.

For any $a,b$ such that $ab = h < N$, we have by induction hypothesis $E(h) \leq c_1(r)(ab)^{1-\frac{1}{r}}p^{\frac{r+1}{4r^2}}(\log p)^{\frac{1}{r}}$. Therefore,
\begin{multline}\label{error}
\frac{1}{c_1(r)p^{\frac{r+1}{4r^2}}(\log{p})^{\frac{1}{r}}}\cdot\frac{2}{H}\sum_{a,b} E(ab) \le\frac{2}{\lfloor A\rfloor \lfloor B\rfloor}\sum_{1\leq a\leq A}\sum_{1 \leq b \leq B} (ab)^{1-\frac{1}{r}}\\
\le 2\frac{1}{AB}\left(\int_1^{A+1} t^{1-\frac{1}{r}}\,dt\right)\left(\int_1^{B+1} t^{1-\frac{1}{r}}\,dt\right)\frac{AB}{(A-1)(B-1)}\\\le
(AB)^{1-\frac{1}{r}}\frac{2}{(2-\frac{1}{r})^2} \left(\frac{(A+1)(B+1)}{AB}\right)^{2-\frac{1}{r}}
\frac{AB}{(A-1)(B-1)}
\\ = \frac{2r^2}{(2r-1)^2}k^{1-\frac{1}{r}} N^{1-\frac{1}{r}}\left(\frac{(A+1)(B+1)}{AB}\right)^{2-\frac{1}{r}}
\frac{AB}{(A-1)(B-1)}.
 \end{multline}

Combining equations \eqref{Opti}, \eqref{WB}, \eqref{log}, \eqref{la otra vez} and \eqref{error} with \eqref{burgess ineq V/H + 2/H} yields (for $r=2$)
\begin{multline}\label{burgess super messy}
\frac{|S_{\chi}(M,N)|}{N^{\frac{1}{2}}p^{\frac{3}{16}}(\log{p})^{\frac{1}{2}}}\leq \frac{AB}{(A-1)(B-1)}\left(12\right)^{\frac{1}{4}}\left(1 + \frac{3}{8k\log{p}}  + \frac{\log{\left(1.85k\log{p}\right)}}{k \log^2{p}}\right)^{\frac{1}{4}} \\
+ \frac{8}{9}k^{\frac{1}{2}}c_1(2) \left(\frac{(A+1)(B+1)}{AB}\right)^{\frac{3}{2}}\frac{AB}{(A-1)(B-1)}. \end{multline}

Similarly, for $r\ge 3$, combining equations \eqref{Opti}, \eqref{WB}, \eqref{log2}, \eqref{la otra vez2} and \eqref{error} with \eqref{burgess ineq V/H + 2/H} yields
\begin{multline}\label{burgess super messy2}
\frac{|S_{\chi}(M,N)|}{N^{1-\frac{1}{r}}p^{\frac{r+1}{4r^2}}(\log{p})^{\frac{1}{r}}}\leq \left(\frac{2r(2r-1)}{r-1}\right)^{\frac{1}{2r}}\big((2r-3)!!(r-1)\big)^{\frac{1}{2r^2}}\frac{AB}{(A-1)(B-1)}
\\ \left(\frac{s^2}{\left((2r-3)!!(r-1)\right)^{\frac{1}{r}}p^{\frac{r-2}{2(r-1)}}} + \frac{1}{4k\log{p}} + \frac{1}{4r(r-1) k \log{p}} + \frac{\log{\left(\frac{1.85s\,k\log{p}}{((2r-3)!!(r-1))^{\frac{1}{r}}}\right)}}{k \log^2{p}}\right)^{\frac{1}{2r}} \\
+ \frac{2r^2}{(2r-1)^2}k^{1-\frac{1}{r}}c_1(r) \left(\frac{(A+1)(B+1)}{AB}\right)^{2-\frac{1}{r}}\frac{AB}{(A-1)(B-1)}. \end{multline}

Now, if we let $c_1(r)$ be defined as follows
\begin{equation}\label{c}
c_1(2) = \frac{AB}{(A-1)(B-1)}\left(12\right)^{\frac{1}{4}}\frac{\left(1 + \frac{3}{8k\log p}  + \frac{\log{\left(1.85k\log{p}\right)}}{k \log^2{p}}\right)^{\frac{1}{4}}}{1 - \frac{8}{9}k^{\frac{1}{2}} \left(\frac{(A+1)(B+1)}{AB}\right)^{\frac{3}{2}}\left(\frac{AB}{(A-1)(B-1)}\right)},
\end{equation}
for $r =2$, and
\begin{multline}\label{c2}
c_1(r) = \frac{AB}{(A-1)(B-1)}\left(\frac{2r(2r-1)\left((2r-3)!! (r-1)\right)^{\frac{1}{r}}}{r-1}\right)^{\frac{1}{2r}}\\
\cdot\frac{\left(\frac{s^2}{\left((2r-3)!!(r-1)\right)^{\frac{1}{r}}p^{\frac{1}{2}-\frac{1}{2r}-\frac{1}{2r(r-1)}}} + \frac{1}{4k\log{p}} + \frac{1}{4r(r-1) k \log{p}} + \frac{\log{\left(\frac{1.85s\,k\log{p}}{((2r-3)!!(r-1))^{\frac{1}{r}}}\right)}}{k \log^2{p}}\right)^{\frac{1}{2r}}}{1 - \frac{2r^2}{(2r-1)^2}k^{1-\frac{1}{r}}\left(\frac{(A+1)(B+1)}{AB}\right)^{2-\frac{1}{r}}\left(\frac{AB}{(A-1)(B-1)}\right)},
\end{multline}
for $r\ge 3$. Therefore from \eqref{burgess super messy} and \eqref{burgess super messy2}, we get that
$$|S_{\chi}(M,N)| \le c_1(r) N^{1-\frac{1}{r}}p^{\frac{r+1}{4r^2}}(\log{p})^{\frac{1}{r}}.$$

All we have to do is pick $k$ to minimize $c_1(r)$ in such a way that $\lfloor A\rfloor\ge 28$, and that $N\ge 12 A$. First, we'll start by showing that
\begin{equation*}\label{burgess B>12}
B \ge 15.
\end{equation*}

Since $B = ((2r-3)!!(r-1))^{\frac{1}{r}}p^{\frac{1}{2r}}$, we can just manually check for $ 2\le r\le 20$  that the inequality is satisfied. To show that it works for $r\ge 21$, we can show that
\begin{equation}\label{machin32}
((2r-3)!!(r-1))^{\frac{1}{r}} \ge 15,
\end{equation}
by noticing that it works for $r=21$ and that the left hand side of \eqref{machin32} is increasing. Indeed, the left hand side is increasing; by noticing that $(2r-3)(r-1) < (2r-1)(r+1)$, we get
$$(2r-3)!!(r-1) < \frac{(2r-1)^{r-1}(r+1)^{r-1}}{(r-1)^{r-2}} < \frac{(2r-1)^{r}(r+1)^{r}}{(r-1)^{r}},$$
implying that
$$\frac{1}{r}\log{((2r-3)!!(r-1))} < \log{((2r-1)(r+1))} - \log{(r-1)},$$
which implies
$$\frac{r+1}{r}\log{((2r-3)!!(r-1))} < \log{((2r-1)!!)} + \log{(r+1)},$$
and hence
$$\log{\left(((2r-3)!!(r-1))^{\frac{1}{r}}\right)} < \log{\left(((2r-1)!!)(r+1))^{\frac{1}{r+1}}\right)}.$$

Using that $B \ge 15$, since $A = \frac{kN}{B}$, then
\begin{equation*}\label{burgess A}
A = \frac{kN}{B} < \frac{kN}{12} < \frac{N}{12},
\end{equation*}
whenever $k < 1$.

Let $k \ge \frac{1}{30}$. To check that $\lfloor A\rfloor \ge 28$, we use \eqref{range2 for N} and \eqref{burgess B} and we note that
\begin{equation*}
\lfloor A\rfloor \ge A-1 \ge\frac{N}{30B}-1 \ge \frac{c_1(r)^r p^{\frac{1}{4}-\frac{1}{4r}}\log{p}}{30((2r-3)!!(r-1))^{\frac{1}{r}}}-1.
\end{equation*}
Table \ref{burgess table lower bounds for c} shows the lower bound $c_1(r)$ must satisfy to have $\lfloor A\rfloor\ge 28$ in different situations.

\begin{table}
\begin{center}
	\begin{tabular}{|c| c | c| c|}
\hline
	$r$ & $p\ge 10^7$ & $p\ge 10^{10}$ & $p\ge 10^{20}$ \\ \hline
$2$ & $2.68289$ & $1.45765$ & $0.24442$ \\  \hline

$3$ & $1.88354$ & $1.13939$ & $0.251637$ \\  \hline

$4$ & $1.6153$ & $1.06881$ & $0.305418$ \\  \hline

$5$ & $1.48379$ & $1.04807$ & $0.363232$ \\  \hline

$6$ & $1.40512$ & $1.04167$ & $0.417191$ \\  \hline

$7$ & $1.35216$ & $1.04007$ & $0.465518$ \\  \hline

$8$ & $1.31369$ & $1.04016$ & $0.508197$ \\  \hline

$9$ & $1.28422$ & $1.04077$ & $0.545749$ \\  \hline

$10$ & $1.26077$ & $1.04147$ & $0.578819$ \\  \hline
	\end{tabular}
\end{center}
\caption{Lower bounds for the constant $c_1(r)$ in the Burgess inequality to satisfy $\lfloor A\rfloor\ge 28$.}\label{burgess table lower bounds for c}
\end{table}

Now, $B$ is defined in terms of $r$ and $p$. By fixing an $r$ and a $p_0$ (a fixed lower bound for $p$), we can calculate $B$ in terms of $r$ and $p_0$. Let $c_1'(r)$ be a parameter satisfying $c_1'(r) < c_1(r)$. $A$ is written in terms of $k$ and $N$ and from \eqref{range for N} and \eqref{range2 for N} we have a range for $N$ in terms of $c_1(r)$, $p$, $r$, $k$ and $s$. From this we can find a lower bound for $A$ in terms of $c_1'(r)$, $k$, $r$, $s$ and $p$. The parameter $s$ is optimal when it is as small as possible so we fix $s$ (in terms of $r$, $c_1'(r)$ and $c_1(r-1)$) to be the smallest real satisfying \eqref{s condition}. After plugging in $A$, $B$, $r$, $s$, $k$, $p_0$ and $c_1'(r)$ to the equations \eqref{c} and \eqref{c2}, we can find a good value of $k \in [\frac{1}{30},1)$, and a good value of $c_1'(r)$ for each $r$ and $p_0$ to find the smallest $c_1(r)$ we can. After making the choices of $k$ and $c_1'(r)$ described in Table \ref{last table 1}, we conclude that $c_1(r)$ has the values listed in Table \ref{super table burgess} as upper bounds.

\begin{table}[h]
\begin{center}
	\begin{tabular}{c| c | c| c| c | c | c |}
\cline{2-7} & \multicolumn{2}{|c|}{$p_0 = 10^7$} & \multicolumn{2}{|c|}{$p_0 = 10^{10}$} & \multicolumn{2}{|c|}{$p_0 = 10^{20}$ }\\ \hline

\multicolumn{1}{|c|}{$r$} & $k$ & $c_1'(r)$ & $k$ & $c_1'(r)$ & $k$ & $c_1'(r)$ \\ \hline

\multicolumn{1}{|c|}{$2$} & 2/45 & 2.738 & 1/30 & 2.517 & 1/30 & 2.354\\  \hline

\multicolumn{1}{|c|}{$3$} & 1/16 & 2.019 & 11/150 & 1.737 & 2/15 & 1.369\\  \hline

\multicolumn{1}{|c|}{$4$} & 1/12 & 1.729 & 31/300 & 1.515 & 37/300 & 1.310\\  \hline

\multicolumn{1}{|c|}{$5$} & 1/12 & 1.610 & 7/75 & 1.456 & 31/300 & 1.298\\  \hline

\multicolumn{1}{|c|}{$6$} & 1/12 & 1.548 & 1/12 & 1.426 & 7/75 & 1.289\\  \hline

\multicolumn{1}{|c|}{$7$} & 11/150 & 1.504 &11/150 & 1.404 & 1/12 & 1.281\\  \hline

\multicolumn{1}{|c|}{$8$} & 19/300 & 1.470 & 19/300 & 1.383 & 1/12 & 1.272\\  \hline

\multicolumn{1}{|c|}{$9$} & 19/300 & 1.441 & 19/300 & 1.366 & 11/150 & 1.264\\  \hline

\multicolumn{1}{|c|}{$10$} & 4/75 & 1.415 & 4/75 & 1.349 & $11/150$ & 1.256\\  \hline
	\end{tabular}
\end{center}
\caption{Values chosen for $k$ and $c_1'(r)$ to build Table \ref{super table burgess}.}\label{last table 1}
\end{table}

\end{proof}

\begin{proof}[Proof of Corollary \ref{burgess corollary 1}]
We begin by pointing out that Theorem \ref{burgess kiks 1} proves this for $2\le r\le 10$ and $p \ge 10^{7}$. We also know that it is true for the $r=1$ case by the P\'olya--Vinogradov inequality (Vinogradov proved it with the constant 1 in \cite{Vinogradov}).

Following the proof of Theorem \ref{burgess kiks 1}, we also have that $B \ge 15$ for all $r$ and hence, for any $k < 1$, we have $A < \frac{N}{12}$. It is also worth pointing out that we can use $s = 1$, since now the constant $2.74$ is fixed as the constant in our upper bound, instead of a constant depending on $r$.

We need to show that you can pick a $k$ such that $\lfloor A\rfloor \ge 28$. First, let's prove that $2.74^r \ge ((2r-3)!!(r-1))^{\frac{1}{r}}$. Indeed, for all $r \ge 1$ we have
$$2.74^r > 2r \ge ((2r -3)!!(r-1))^{\frac{1}{r}}.$$

Now we have
$$A = \frac{kN}{B} \ge \frac{k (2.74)^r p^{\frac{1}{4}-\frac{1}{4r}}\log{p}}{((2r-3)!!(r-1))^{\frac{1}{r}}} \ge k p^{\frac{1}{4}-\frac{1}{4r}}\log{p} > 29,$$
whenever $k >  \frac{29}{p^{\frac{1}{4}-\frac{1}{4r}}\log{p}}$.

We replace $B$ by 15 in \eqref{c2} (since $B \ge 15$), and we can see that the only factors that don't decrease with $r$ are the $k^{1-\frac{1}{r}}$ term which appears in the denominator, and the $\left(2-\frac{1}{r}\right)$ exponent in the denominator. With this in mind, let $c(r)$ be defined as follows for $r \ge 4$:
\begin{multline}\label{cmatad}
c(r) = \frac{15A}{14(A-1)}\left(\frac{2r(2r-1)\left((2r-3)!! (r-1)\right)^{\frac{1}{r}}}{r-1}\right)^{\frac{1}{2r}}\\
\cdot\frac{\left(\frac{1}{\left((2r-3)!!(r-1)\right)^{\frac{1}{r}}p^{\frac{1}{2}-\frac{1}{2r}-\frac{1}{2r(r-1)}}} + \frac{1}{4k\log{p}} + \frac{1}{4r(r-1) k \log{p}} + \frac{\log{\left(\frac{1.85\,k\log{p}}{((2r-3)!!(r-1))^{\frac{1}{r}}}\right)}}{k \log^2{p}}\right)^{\frac{1}{2r}}}{1 - \frac{2r^2}{(2r-1)^2}k\left(\frac{16(A+1)}{15A}\right)^2\left(\frac{15A}{14(A-1)}\right)}.
\end{multline}

Letting $k = \frac{11}{64}$, $A \ge kp^{\frac{1}{4}-\frac{1}{4r}}$ and $p \ge 10^7 $ we confirm that $c(r) \le 2.74$ whenever $r\ge 3$. Since it is also true for $r\le3$, we conclude our corollary.
\end{proof}

\section{Improving McGown's theorem}\label{burgess section 2}
The main obstacle in improving the $(\log{p})^{\frac{1}{r}}$ factor in the Burgess inequality is the bound on $V_2$. However, if we put a bound on $N$, we can make the proof cleaner while also improving the exponent in $\log{p}$ to $\frac{1}{2r}$. First we prove a lemma regarding $V_2$ and then we will be able to prove Theorem \ref{burgess kiks 2}.

\begin{lemma}\label{burgess V2 improved}
Let $p$ be a prime, and $N$ be a positive integer. Let $A \ge 30$ be an integer such that $N>7A$ and $2AN< p$. Let $v(x)$ be defined as in \eqref{burgess v(x)}. Then
\begin{equation*}\label{burgess V2 eq}
V_2 = \sum_{x\kern-3pt\mod{p}}v^2(x) \leq 2AN \log(1.85 A).
\end{equation*}
\end{lemma}

\begin{proof}
The proof is essentially the same as that of Lemma \ref{burgess V2 lemma}. Recall that $V_2$ is the number of quadruples $(a_1,a_2,n_1,n_2)$ with $1\leq a_1,a_2 \leq A$ and $M < n_1,n_2 \leq M + N$ such that $a_1n_2 \equiv a_2n_1\pmod p$. If $a_1 = a_2$, since $N < p$, we have that $n_1 = n_2$ because $n_1 \equiv n_2 \pmod p$ while $|n_1 - n_2| \leq N < p$. Therefore, the number of quadruples in this case is $AN$. Fixing $a_1\ne a_2$ and writing
$$a_1n_2-a_2n_1 = kp,$$
we can put a bound on possible values for $k$.  As shown in the proof of Lemma \ref{burgess V2 lemma}, there are at most $\frac{(a_1+a_2)N}{\gcd{(a_1,a_2)}p}  + 1$ values of $k$. Since $2AN < p$, then we have that $k$ is uniquely determined.

In the proof of Lemma \ref{burgess V2 lemma}, we showed that given $a_1,a_2$ and $k$, the number of pairs $(n_1,n_2)$ is bounded by $N\frac{\gcd{(a_1,a_2)}}{\max\{a_1,a_2\}} + 1$.

Now, for $A\ge 30$ and $N > 7A$ we have
\begin{equation}\label{orales}
\left(1-\frac{6}{\pi^2}\right)\log{(1.85 A)} \ge \left(1-\frac{6}{\pi^2}\right)\log{(55.5)} = 1.57471 > \frac{3}{2} + \frac{1}{14} > \frac{3}{2} + \frac{A}{2N}.
\end{equation}

Using the definition of $S_3$ as in \eqref{S_3}, using the inequalities \eqref{burgess S3} and \eqref{orales}, for $A\ge 30$ and $N>7A$, we have
\begin{multline*}
V_2 \le AN + 2\sum_{a_1 < a_2}\left(\frac{\gcd{(a_1,a_2)}N}{\max\{a_1,a_2\}} + 1\right)\\
= AN + 2NS_3 + A^2-A \le 2AN\log{(1.85A)}.
\end{multline*}
\end{proof}
Now we are ready to prove Theorem \ref{burgess kiks 2}.

\begin{proof}[Proof of Theorem \ref{burgess kiks 2}]
The proof is very similar to the proof of Theorem \ref{burgess kiks 1}. We proceed by induction, assuming that for all $h < N$ we have $|S_{\chi}(M,h)|\le c_2(r) p^{\frac{r+1}{4r^2}}(\log{p})^{\frac{1}{2r}}.$

Most of the work in the proof of Theorem $\ref{burgess kiks 1}$ can be replicated. So I'll just point out the things that change.

The first change is that by employing Lemma \ref{burgess V2 improved}, \eqref{burgess V2} becomes
\begin{equation*}\label{burgess V2 2}
V_2 \le 2AN\log{(1.85A)}.
\end{equation*}
This change affects \eqref{Opti}, by deleting $\frac{AN}{p}$ inside the parenthesis. Now it looks as follows:
\begin{equation}\label{burgess V/H3}
\frac{V}{H} \le \frac{AB}{(A-1)(B-1)}\frac{1}{k^{\frac{1}{2r}}}\frac{(2WB)^{\frac{1}{2r}}}{B}(\log{(1.85A)})^{\frac{1}{2r}}
N^{1-\frac{1}{r}}.
\end{equation}

The next change is the range for $N$, which we deduced by using the P\'olya--Vinogradov inequality, the trivial bound, and the case for $r-1$. Instead of \eqref{range for N}, using our hypothesis and the trivial bound, we now have
\begin{equation}\label{new range for N}
c_2(2)^r p^{\frac{3}{8}} \sqrt{\log{p}} < N < 2 p^{\frac{5}{8}},
\end{equation}
for $r=2$. Assuming $c_2(r-1) \le s^{\frac{1}{r(r-1)}}c_2(r)$ for a real number $s$, and using the Burgess inequality for $r-1$ we have, for $r\ge 3$, the following range for $N$
\begin{equation}\label{new range for N2}
c_2(r)^r p^{\frac{1}{4} + \frac{1}{4r}} \sqrt{\log{p}} < N < \min\{2 p^{\frac{1}{2} + \frac{1}{4r}},s\,p^{\frac{1}{4}+\frac{1}{2r}+\frac{1}{4r(r-1)}}\sqrt{\log{p}}\}.
\end{equation}

Using that $A=\frac{kN}{B}$ and \eqref{new range for N}, we get
\begin{equation}\label{la otra vez 2}
\log{(1.85A)}=\log{\left(\frac{1.85kN}{B}\right)}\le \log{\left(3.7k\right)}+\frac{3\log{p}}{8} \le \frac{3\log{p}}{8},
\end{equation}
for $r=2$ (we're assuming $k < 1/4$, which implies $\log{(3.7k)} < 0$). Using \eqref{new range for N2}, yields
\begin{equation}\label{la otra vez 22}
\log{(1.85A)}=\log{\left(\frac{1.85kN}{B}\right)}\le \log{\left(\frac{1.85s\,k\sqrt{\log{p}}}{((2r-3)!!(r-1))^{\frac{1}{r}}}\right)}+\frac{\log{p}}{4} + \frac{\log{p}}{4r(r-1)},
\end{equation}
for $r\ge 3$.

The bound for $E(h)$ is almost the same as in \eqref{error}, the only difference being the exponent of $\log{p}$, which is now $\frac{1}{2r}$ instead of $\frac{1}{r}$. Making this change and using both \eqref{WB} and \eqref{la otra vez 2} with \eqref{burgess V/H3} yields (for $r=2$)
\begin{multline}\label{burgess super messy 2}
\frac{|S_{\chi}(M,N)|}{N^{\frac{1}{2}}p^{\frac{3}{16}}(\log{p})^{\frac{1}{4}} }\le \frac{AB}{(A-1)(B-1)}\left(12\right)^{\frac{1}{4}} \left(\frac{3}{8k} \right)^{\frac{1}{4}} \\
+ \frac{AB}{(A-1)(B-1)}\frac{8}{9}k^{\frac{1}{2}}c_2(2) \left(\frac{(A+1)(B+1)}{AB}\right)^{\frac{3}{2}}. \end{multline}

For $r\ge 3$, using \eqref{WB} and \eqref{la otra vez 22} with \eqref{burgess V/H3} yields
\begin{multline}\label{burgess super messy 22}
\frac{|S_{\chi}(M,N)|}{N^{1-\frac{1}{r}}p^{\frac{r+1}{4r^2}}(\log{p})^{\frac{1}{2r}} }\le \frac{AB}{(A-1)(B-1)}\left(\frac{2r(2r-1)\left((2r-3)!!(r-1)\right)^{\frac{1}{r}}}{r-1}\right)^{\frac{1}{2r}}
\\
\cdot \left(\frac{\log{\left(\frac{1.85s\,k\sqrt{\log{p}}}{((2r-3)!!(r-1))^{\frac{1}{r}}}\right)}}{k\log{p}}+\frac{1}{4k} +\frac{1}{4r (r-1)k}\right)^{\frac{1}{2r}} \\
+ \frac{AB}{(A-1)(B-1)}\frac{2r^2}{(2r-1)^2}k^{1-\frac{1}{r}}c_2(r) \left(\frac{(A+1)(B+1)}{AB}\right)^{2-\frac{1}{r}}. \end{multline}

Now, if we let $c_2(r)$ be defined as follows
\begin{equation}\label{c21}
c_2(2) = \frac{A}{A-1}\frac{B}{B-1}\frac{\left(12\right)^{\frac{1}{4}}\left(\frac{3}{8k}\right)^{\frac{1}{4}}}{1 - \frac{8}{9}k^{\frac{1}{2}}\left(\frac{(A+1)(B+1)}{AB}\right)^{\frac{3}{2}}\left(\frac{AB}{(A-1)(B-1)}\right)},
\end{equation}
and, for $r\ge 3$,
\begin{equation}\label{c22}
c_2(r) = \frac{A}{A-1}\frac{B}{B-1}\frac{\left(\frac{2r(2r-1)\left((2r-3)!! (r-1)\right)^{\frac{1}{r}}}{r-1}\left(\frac{\log{\left(\frac{1.85s\,k\sqrt{\log{p}}}{((2r-3)!!(r-1))^{\frac{1}{r}}}\right)}}{k\log{p}}+\frac{1}{4k} +\frac{1}{4r (r-1)k}\right)\right)^{\frac{1}{2r}}}{1 - \frac{2r^2}{(2r-1)^2}k^{1-\frac{1}{r}}\left(\frac{(A+1)(B+1)}{AB}\right)^{2-\frac{1}{r}}\left(\frac{AB}{(A-1)(B-1)}\right)},
\end{equation}
for $r\ge 3$. Then, from \eqref{burgess super messy 2} and \eqref{burgess super messy 22}, we get that
$$|S_{\chi}(M,N)| \le c_2(r) N^{1-\frac{1}{r}}p^{\frac{r+1}{4r^2}}(\log{p})^{\frac{1}{2r}}.$$

All we have to do is pick $k$ to minimize $c_2(r)$ in such a way that $\lfloor A\rfloor\ge 30$, that $N> 7 A$ and $2AN < p$.

Using that $B \ge 15$, it is not hard to check that $N\ge 7A$. Indeed, since $A = \frac{kN}{B}$, we have $A \le \frac{kN}{15} < \frac{N}{7}$.

To check that $\lfloor A\rfloor \ge 30$ for $k\ge \frac{3}{64}$, we do the following:
\begin{equation*}
\lfloor A\rfloor\ge A-1 \ge \frac{3N}{64B} -1 \ge \frac{3c_2(r)^r p^{\frac{1}{4}-\frac{1}{4r}}\sqrt{\log{p}}}{64((2r-3)!!(r-1))^{\frac{1}{r}}}-1.
\end{equation*}
Table \ref{burgess table lower bounds for c 2} shows the lower bound $c$ must satisfy to have $\lfloor A\rfloor\ge 30$ in different situations.

\begin{table}[h]
\begin{center}
	\begin{tabular}{|c| c | c| c|}
\hline
	$r$ & $p\ge 10^{10}$ & $p\ge 10^{15}$ & $p\ge 10^{20}$ \\ \hline
$2$ & $2.78392$ & $1.22500$ & $0.55514$ \\  \hline

$3$ & $1.75393$ & $0.86474$ & $0.43480$ \\  \hline

$4$ & $1.47708$ & $0.81850$ & $0.46029$ \\  \hline

$5$ & $1.35767$ & $0.82260$ & $0.50431$ \\  \hline

$6$ & $1.29240$ & $0.83775$ & $0.54839$ \\  \hline

$7$ & $1.25127$ & $0.85450$ & $0.58848$ \\  \hline

$8$ & $1.22279$ & $0.87022$ & $0.62388$ \\  \hline

$9$ & $1.20171$ & $0.88422$ & $0.65489$ \\  \hline

$10$ & $1.18536$ & $0.89649$ & $0.68202$ \\  \hline
	\end{tabular}
\end{center}
\caption{Lower bounds for the constant $c_2(r)$ in the Burgess inequality to satisfy $\lfloor A\rfloor\ge 30$.}\label{burgess table lower bounds for c 2}
\end{table}

Let's now verify that $2AN < p$. Indeed, from the fact that $A=\frac{kN}{B}$ and from \eqref{new range for N}, we have
\begin{equation*}
2AN = \frac{2kN^2}{B} \le \frac{8kp}{((2r-3)!!(r-1))^{\frac{1}{r}}} < p,
\end{equation*}
whenever $k < \min\left\{\frac{((2r-3)!!(r-1))^{\frac{1}{r}}}{8},1\right\}.$

As in the proof of Theorem 1, we define can find bounds for $c_2(r)$ by controlling the parameters $p_0$, $c_2'(r)$ and $k$. We find a good value of $k \in [\frac{3}{64},\frac{((2r-3)!!(r-1))^{\frac{1}{r}}}{8})$ and a good value of $c_2'(r)$ for each $r$ and $p_0$, and plug in the values of $B$, $k$, and a lower bound bound for $A$ on \eqref{c21} to find $c_2(2)$ and on \eqref{c22} to find $c_2(r)$ for $r\ge 3$ in Table \ref{table kiks burgess} and conclude the theorem. The values of $k$ and $c_2'(r)$ we chose can be found on Table \ref{last table 2}.

\begin{table}[h]
\begin{center}
	\begin{tabular}{c| c | c| c| c | c | c |}
\cline{2-7} & \multicolumn{2}{|c|}{$p_0 = 10^{10}$} & \multicolumn{2}{|c|}{$p_0 = 10^{15}$} & \multicolumn{2}{|c|}{$p_0 = 10^{20}$ }\\ \hline

\multicolumn{1}{|c|}{$r$} & $k$ & $c_2'(r)$ & $k$ & $c_2'(r)$ & $k$ & $c_2'(r)$ \\ \hline

\multicolumn{1}{|c|}{$2$} & 0.124 & 3.65 & 0.124 & 3.58 & 0.124 & 3.57\\  \hline

\multicolumn{1}{|c|}{$3$} & 0.126 & 2.58 & 0.131 & 2.51 & 0.135 & 2.49\\  \hline

\multicolumn{1}{|c|}{$4$} & 0.106 & 2.19 & 0.116 & 2.12 & 0.120 & 2.10\\  \hline

\multicolumn{1}{|c|}{$5$} & 0.091 & 1.98 & 0.101 & 1.92 & 0.107 & 1.90\\  \hline

\multicolumn{1}{|c|}{$6$} & 0.080 & 1.85 & 0.090 & 1.79 & 0.095 & 1.77\\  \hline

\multicolumn{1}{|c|}{$7$} & 0.072 & 1.75 & 0.079 & 1.70 & 0.084 & 1.68\\  \hline

\multicolumn{1}{|c|}{$8$} & 0.064 & 1.68 & 0.071 & 1.635 & 0.077 & 1.61\\  \hline

\multicolumn{1}{|c|}{$9$} & 0.058 & 1.625 & 0.065 & 1.58 & 0.070 & 1.56\\  \hline

\multicolumn{1}{|c|}{$10$} & 0.054 & 1.579 & 0.060 & 1.54 & 0.064 & 1.52\\  \hline
	\end{tabular}
\end{center}
\caption{Values chosen for $k$ and $c_2'(r)$ to build Table \ref{table kiks burgess}.}\label{last table 2}
\end{table}
\end{proof}

\begin{proof}[Proof of Corollary \ref{burgess kiks corollary 2}]
By Theorem \ref{burgess kiks 2}, we have our desired result whenever $N < 2 p^{\frac{1}{2} + \frac{1}{4r}}$. Therefore, the only thing we need to prove is that for $p\ge 10^{10}$ and $r\ge 3$, $N < 2 p^{\frac{1}{2}+\frac{1}{4r}}$. Since the induction in the proof of Theorem \ref{burgess kiks 2} relied on the upper bound for $N$, we can't use the Burgess inequalities in Theorem \ref{burgess kiks 2} to give an upper bound for $N$ in this corollary. However, we can use the Burgess inequalities from Theorem \ref{burgess kiks 1} to improve the upper bound for $N$. Indeed, for $p\ge 10^{10}$, we have
$$|S_{\chi}(M,N)| \le 2.6 N^{1-\frac{1}{2}}p^{\frac{3}{16}}(\log{p})^{\frac{1}{2}}.$$
If
$$N \ge (2.6)^{\frac{2r}{r-1}}p^{\frac{3}{8}-\frac{1}{8r} - \frac{3}{8r(r-1)}}\sqrt{\log{p}},$$
then
$$N^{1-\frac{1}{r}}p^{\frac{r+1}{4r^2}}(\log{p})^{\frac{1}{2r}} \ge 2.6 N^{1-\frac{1}{2}}p^{\frac{3}{16}}\sqrt{\log{p}} \ge |S_{\chi}(M,N)|.$$

Therefore, we may assume that
\begin{equation}\label{burgess corollary 2 eq}
N \le (2.6)^{\frac{2r}{r-1}}p^{\frac{3}{8}-\frac{1}{8r} - \frac{3}{8r(r-1)}}\sqrt{\log{p}}.
\end{equation}

Now, all we need to conclude is to show that the right hand side of \eqref{burgess corollary 2 eq} is less than $2p^{\frac{1}{2} +\frac{1}{4r}}$. Using that $p \ge 10^{10}$, we can verify this manually for $r\in\{3,4,\ldots,21\}$. Now, for $r\ge 22$ we have
$$N \le (2.6)^{\frac{2r}{r-1}}p^{\frac{3}{8} -\frac{1}{8r} - \frac{3}{8r(r-1)}}\sqrt{\log{p}} \le (2.6)^{\frac{44}{21}}p^{\frac{3}{8}}\sqrt{\log{p}} < 2p^{\frac{1}{2}}.$$
The last inequality is true whenever $p\ge 10^{10}$.
\end{proof}

\begin{remark}
Booker and McGown in their proofs have $A$ range through only prime numbers. This idea makes the constants converge quicker. For large enough $p$, it doesn't improve the numbers, but it does for smaller $p$. To save space, we ommited using that technique here, instead focusing on other techniques that made an impact on the "asymptotic" constant. One of the nice ideas not used by McGown or Booker is the idea of using Burgess for smaller $r$ to help out with the larger $r$. This allows the theorems to extend to the whole range when $r \ge 3$.
\end{remark}

\begin{remark}
Theorem A is a little stronger in \cite{ETk} when the order of the character is bigger. Therefore, one could use that theorem to get better constants for cubic characters, quartic characters and so on.
\end{remark}

\section{Least $k$-th power non-residue}\label{section p^1/6}
To prove our results on the least $k$-th power non-residues, we will need the following estimates from \cite{RS1962}:

\begin{lemma}\label{Rosser and Scho}
Let $B = \displaystyle\lim_{m\rightarrow\infty} \sum_{p\le m}\frac{1}{p} - \log{\log{x}}$, and let $\pi(x)$ be the number of primes up to $x$. Then the following estimates are true:
\begin{align*}
\log{\log{x}} + B - \frac{1}{2\log^2{x}} < \sum_{p\le x} \frac{1}{p} & \mbox{ for } x > 1,\\
\sum_{p\le x} \frac{1}{p} < \log{\log{x}} + B + \frac{1}{2\log^2{x}} & \mbox{ for } x \ge 286,\\
\pi(x) < \frac{x}{\log{x}}\left(1 + \frac{3}{2\log{x}}\right) & \mbox{ for } x > 1.
\end{align*}
\end{lemma}
From it we derive the following immediate corollary:
\begin{corollary}\label{Vino corollary}
For real numbers $x,y$ satisfying $x > y > 1$ and $x \ge 286$, the following estimate is true:
\begin{equation*}
\sum_{y < p\le x} \frac{1}{p} < \log{\log{x}}-\log{\log{y}} + \frac{1}{2\log^2{x}} + \frac{1}{2\log^2{y}}.
\end{equation*}
\end{corollary}

Now we are ready to prove the key lemma (a lower bound on a character sum), which is the essence of Vinogradov's trick.
\begin{lemma}\label{lower bound non-residue}
Let $x\ge 286$ be a real number, and let $ y = x^{\frac{1}{\sqrt{e}} + \delta}$ for some $\delta > 0$. Let $\chi$ be a non-principal character $\bmod{\,p}$ for some prime $p$. If $\chi(n) = 1$ for all $n\le y$, then
\begin{equation*}
\left|\sum_{n\le x} \chi(n)\right| \ge x\left(2\log{(\delta\sqrt{e}+1)} - \frac{1}{\log^2{x}} - \frac{1}{\log^2{y}} - \frac{1}{x}\right).
\end{equation*}
\end{lemma}
\begin{proof}
Since $\chi(n)$ is totally multiplicative, $\chi(n) = 1$ for all $n \le y$, and $x < p$, then
\begin{align*}\label{vinogradov trick step 1}
\left|\sum_{n\le x} \chi(n)\right|
&= \left|\sum_{\substack{n\le x\\ \chi(n)=1}}\chi(n) + \sum_{\substack{n\le x\\ \chi(n)\ne 1}}\chi(n)\right|\\
&= \left|\sum_{n\le x} 1 -\sum_{\substack{n\le x\\ \chi(n)\ne 1}}1 + \sum_{\substack{n\le x\\ \chi(n)\ne 1}}\chi(n)\right|\\
&\ge \sum_{n\le x} 1 - 2\sum_{\substack{n \le x\\ \chi(n) \ne 1}}\chi(n)\\
&\ge \sum_{n\le x} 1 - 2\sum_{\substack{y < q \le x\\ \chi(q) \ne 1}}\sum_{n\le \frac{x}{q}}1,
\end{align*}
where the sum ranges over $q$ prime. Therefore we have
\begin{equation*}
\left|\sum_{n\le x} \chi(n)\right| \ge \left\lfloor x\right\rfloor - 2\sum_{y < q \le x} \left\lfloor\frac{x}{q}\right\rfloor \ge x - 1 - 2x\sum_{y < q \le x} \frac{1}{q} .
\end{equation*}
Using Corollary \ref{Vino corollary} to estimate the sum of the reciprocals of primes we get the desired inequality.
\end{proof}

We can now prove Theorem 3. We will use the explicit Burgess inequality proved as Corollary \ref{burgess corollary 1} because it works for all $r$.
\begin{proof}[Proof of Theorem \ref{theorem p^1.6}]
Let $\chi$ be a character $\bmod{\,p}$. Then if $n < p$ and $\chi(n) \neq 1$, $n$ is a $k$-th power non-residue. Let $r$ be an integer. Let $x\ge 286$ be a real number and let $y = x^{\frac{1}{\sqrt{e}}+\delta} = p^{1/6}$ for some $\delta>0$. Assume that $\chi(n) = 1$ for all $n\le y$. Now by Corollary \ref{burgess corollary 1} and Lemma \ref{lower bound non-residue} we have
\begin{equation*}
2.74x^{1-\frac{1}{r}} p^{\frac{r+1}{4r^2}}(\log{p})^{\frac{1}{r}} \ge x\left(2\log{(\delta\sqrt{e}+1)} - \frac{1}{\log^2{x}} - \frac{1}{\log^2{y}} - \frac{1}{x}\right).
\end{equation*}
Now, letting $x = p^{\frac{1}{4} + \frac{1}{2r}}$ we get
\begin{equation}\label{p^1/6}
2.74 p^{\frac{\log{\log{p}}}{r \log{p}} - \frac{1}{4r^2}} \ge 2\log{(\delta\sqrt{e}+1)} - \frac{1}{\log^2{x}} - \frac{1}{\log^2{y}} - \frac{1}{x}.
\end{equation}
Picking $r = 22$, one finds that $\delta = 0.00458\ldots$. For $p \ge 10^{4732}$, the right hand side of \eqref{p^1/6} is bigger than the left hand side, showing that $\chi(n)$ is not always 1 for $n \le y = p^{1/6}$, and hence the theorem is true.
\end{proof}

\begin{remark}
To be able to use Theorem \ref{burgess kiks 2} to improve Theorem \ref{theorem p^1.6}, we would need to calculate what happens for $r > 20$ since the restriction $y = p^{1/6}$ implies $r > 20$. Since we know that $p$ will be large, we can also pick a large $p_0$ and then find a good constant for the Burgess inequality when $p$ is very large and $r > 20$. After doing all of this work, one could show that Theorem \ref{theorem p^1.6} works for $p \ge 10^{3850}$.
\end{remark}
\nocite{Fri1987}
\nocite{PPP}

\section*{Acknowledgements} I would like to dedicate this paper to Paul Bateman and Heini Halberstam. I would like to thank Carl Pomerance for suggesting the problem, his useful comments and his encouragement. I would like to express my gratitude to Andrew Granville, Hugh Montgomery and Robert Vaughan for allowing me to look at unpublished manuscripts that helped me on my work on the Burgess inequality. I would also like to thank Paul Pollack who directed me to McGown's work. Finally, I'd like to thank an anonymous referee that read the paper carefully and suggested many improvements.

\bibliographystyle{amsplain}
\bibliography{bsample}

\end{document}